\documentclass[reqno,12pt]{amsart}

\usepackage{amsmath,amsthm,amssymb,epsfig}
\usepackage[table]{xcolor}
\usepackage{url}
\usepackage{tikz}
\usepackage{float}
\usepackage{mathtools}
\usepackage{subcaption}
\usepackage{algorithm}
\usepackage{algorithmic}
\usepackage{MnSymbol}
\usepackage{graphicx,xcolor}
\usepackage{comment}
\usepackage{tkz-graph}
\usetikzlibrary{fit}
\usetikzlibrary{shapes}
\usepackage{comment}
\usepackage{enumitem}
\usepackage[left=3.5cm,top=3.5cm,right=3.5cm,bottom=3.5cm]{geometry}

\newcommand{\mcs}{\mathrm{mcs}}

\newtheorem{theorem}{Theorem}
\newtheorem{corollary}[theorem]{Corollary}
\newtheorem{lemma}[theorem]{Lemma}

\title{How to burn a Latin square}

\author[A.\ Bonato]{Anthony Bonato}
\author[C.\ Jones]{Caleb Jones}
\author[T.G.\ Marbach]{Trent G.\ Marbach}
\author[T.\ Mishura]{Teddy Mishura}
\address[A1,A2,A3,A4]{Toronto Metropolitan University, Toronto, Canada}

\email[A1]{(A1) abonato@torontomu.ca}
\email[A2]{(A2) caleb.w.jones@torontomu.ca}
\email[A3]{(A3) trent.marbach@torontomu.ca}
\email[A4]{(A4) tmishura@torontomu.ca}

\begin{document}

\begin{abstract}

We investigate the lazy burning process for Latin squares by studying their associated hypergraphs. In lazy burning, a set of vertices in a hypergraph is initially burned, and that burning spreads to neighboring vertices over time via a specified propagation rule. The lazy burning number is the minimum number of initially burned vertices that eventually burns all vertices. The hypergraphs associated with Latin squares include the $n$-uniform hypergraph, whose vertices and hyperedges correspond to the entries and lines (that is, sets of rows, columns, or symbols) of the Latin square, respectively, and the $3$-uniform hypergraph, which has vertices corresponding to the lines of the Latin square and hyperedges induced by its entries.

Using sequences of vertices that together form a vertex cover, we show that for a Latin square of order $n$, the lazy burning number of its $n$-uniform hypergraph is bounded below by $n^2-3n+3$ and above by $n^2-3n+2 + \lfloor \log_2 n \rfloor.$ These bounds are shown to be tight using cyclic Latin squares and powers of intercalates. For the $3$-uniform hypergraph case, we show that the lazy burning number of Latin squares is one plus its shortest connected chain of subsquares. We determine the lazy burning number of Latin square hypergraphs derived from finitely generated groups. We finish with open problems.
\end{abstract}

\maketitle

\section{Introduction}\label{introduction}

Graph burning is a simplified model for the spread of influence in a network. Associated with the process is a parameter introduced in \cite{first_paper,BJR0}, the burning number, which quantifies the speed at which the influence spreads to every vertex. Given a graph $G$, the burning process on $G$ is a discrete-time process. At the beginning of the first round, all vertices are unburned. In each round, first, all unburned vertices that have a burned neighbor become burned, and then one new unburned vertex is chosen to burn if such a vertex is available. The \emph{burning number} of $G,$ written $b(G),$ is defined to be the least $k$ such that after $k$ rounds, each vertex of $G$ is burned. For further background on graph burning, see the survey \cite{survey} and the book \cite{bbook}.

How do we burn a Latin square? The \emph{Latin square graph} of a Latin square $L$ of order $n,$ is the graph with $n^2$ vertices labeled with the cells of the Latin square, where distinct vertices are adjacent if they share a row, column, or symbol. The burning number of a Latin square graph is at most 3, so it is less interesting. Our approach, instead, is to consider burning various hypergraphs constructed from Latin squares. We will consider propagation rules for burning on hypergraphs that lead to more subtle results. 

First, we provide some background and notations on Latin squares and hypergraphs. An \emph{entry} of a Latin square is a triple $(r,c,s)$, where symbol $s$ occurs in row $r$ and column $c$. We write $(r,c,s)\in L$ if $(r,c,s)$ is an entry of $L$ and denote the entry-set of $L$ as $E(L)$, respectively. A \emph{row}-\emph{line} of a Latin square of order $n$ is the $n$-set of entries in $L$ in a particular row. \emph{Column}- and \emph{symbol}-\emph{lines} are defined similarly. We refer to the row-lines, column-lines, and symbol-lines collectively as the \emph{lines} of $L$ and define $\mathcal{L}(L)$ to be the set of all lines of $L$. 
A \emph{subsquare} of a Latin square $L$ is a subset of the entries of $L$ such that $L'$ is a Latin square in its own right. 
A subsquare $L'$ of $L$ is \emph{trivial} if $L'=L$ or if $L'$ is a single entry, and is \emph{proper}, otherwise. 

Let $L$ be a Latin square of order $n.$ The $n$-uniform hypergraph of $L$, written $H_L$, has vertices $E(L)$ and hyperedges $\mathcal{L}(L)$. This gives an $n$-uniform hypergraph with $n^2$ vertices and $3n$ hyperedges. A hyperedge of $H_L$ is called a \emph{row-edge} if it corresponds to a row-line of the Latin square. The \emph{column-edges} and \emph{symbol-edges} are defined similarly. The $3$-uniform hypergraph of $L$, written $H^L$, is the hypergraph with vertex-set $\mathcal{L}(L)$, and whose hyperedges are the elements of $E(L)$, where we think of an entry $(r,c,s)\in E(L)$ as the set containing the three lines $r$, $c$, and $s$. This gives a $3$-uniform tripartite hypergraph with $3n$ vertices and $n^2$ hyperedges. The hypergraphs $H_L$ and $H^L$ are dual hypergraphs since swapping the roles of vertices and hyperedges in one yields the other.  

\emph{Hypergraph burning} was introduced in \cite{our_paper!,mythesis} as a natural extension of graph burning. The rules for hypergraph burning are identical to those of burning graphs, except for how burning propagates within a hyperedge. In hypergraphs, the burning spreads to a vertex $v$ in round $r$ if and only if $v$ was not burned by the end of round $r-1$ and there is a non-singleton hyperedge $\{v,u_1,\ldots,u_k\}$ such that each of $u_1,u_2,\ldots,u_k$ was burned by the end of round $r-1$. That is, a vertex becomes burned if it is the only unburned vertex in a hyperedge. The definition of burning in hypergraphs, therefore, generalizes burning in graphs. 

A natural variation of burning that is our principal focus here is \emph{lazy hypergraph burning} (also introduced in \cite{our_paper!,mythesis}), where a set of vertices is chosen to burn in the first round; no other vertices are chosen to burn in later rounds.  The \textit{lazy burning number} of $H$, denoted $b_L(H)$, is the minimum cardinality of a set of vertices burned in the first round that eventually burn all vertices. Note that for a hypergraph $H,$ we have that $b_L(H) \le b(H).$ We refer to the set of vertices chosen to burn in the first round as a \emph{lazy burning set}. A \emph{minimum} lazy burning set $S$ is one with $|S|=b_L(H)$. If a vertex is burned and is not part of a lazy burning set, then we say it is burned by \emph{propagation}; all rounds other than the first where a lazy burning set is chosen are called \emph{propagation rounds}.

For an example, consider the hypergraph $H_L$ corresponding to the following Latin square of order $3$, where the darker shaded entries are part of the lazy burning set.

\medskip
    
    \begin{center}\label{fig:burning_example}
    \begin{tabular}{|c|c|c|}
    \hline
         \cellcolor{black!50}1& \cellcolor{black!50}2& \cellcolor{black!20}3\\ \hline
         \cellcolor{black!50}2&3&1\\ \hline
         \cellcolor{black!20}3&1&\cellcolor{black!20}2 \\\hline
    \end{tabular}
    \end{center}

\medskip
\noindent The lighter shaded entries will burn in the first propagation round, as each of them is a member of a hyperedge with two burned vertices. 
This process repeats each round until there are no more vertices to burn; here, the entire hypergraph is burned by the end of the second propagation round.
 
In Theorem~\ref{thm:mcs_equality_FuncOf_stc}, we show that the maximum length of a cover-sequence in a Latin square is determined by the length of its shortest connected chain. We use this in Corollary~\ref{cor:bounds_Bl} to show that $n^2-3n+3 \le b_L(H_L) \le n^2-3n+2 + \lfloor \log_2 n \rfloor,$ and these bounds are shown to be tight using cyclic Latin squares and powers of intercalates. 
In the 3-uniform hypergraph case, $b_L(H^L)$ is shown to be one plus the length of $L$s smallest connected chain in Theorem~\ref{thm:3unistcupper}. Hence, we show that $b_L(H_L)$ and $b_L(H^L)$ differ by a simple parameter of the hypergraph. 
For the case of finitely generated groups, the lazy burning number of the associated 3-uniform hypergraphs is shown as the number of generators plus 2; see Corollary~\ref{dim_plus_two}. We finish with open problems.

For more background on Latin squares, see \cite{designs_handbook,keed,design_text}. For more background on hypergraphs, see \cite{berge1,berge2,bretto}. 

\section{Burning the $n$-uniform case}

We present results on the lazy burning number of two Latin square hypergraph constructions: the $n$-uniform $H_L$ and the 3-uniform $H^L$. We additionally introduce the purely design-theoretic ideas of Latin square \emph{cover-sequences} and \emph{connected chains}, which, while originally developed as machinery for bounds on $b_L(H_L)$ and $b_L(H^L)$, we consider independently interesting. These concepts are tightly linked to the lazy burning number of both hypergraph constructions, which we show below, beginning with the $n$-uniform case.

Intuitively, a good lazy burning set for a Latin square $L$ should be efficient, including multiple rows, columns, and symbols per entry. This is closely related to the idea of a \emph{cover} of $H_L$ (or $L$), which is a set of entries $\{e_1, e_2,\ldots, e_d\}$ such that each row, column, and symbol of $L$ is represented at least once in the set of entries. 

Covers for Latin squares were introduced in \cite{MR3942279} as a way to study partial transversals. There is an extensive literature surrounding transversals and partial transversals. The famous Ryser-Brualdi-Stein conjecture asks whether all even Latin squares have a partial transversal of size $n-1$ and all odd Latin squares have a transversal. A number of authors have contributed to this problem \cite{trans1, trans2, trans3, trans4, trans5, trans6}, and a recent preprint claims to have solved the problem \cite{montgomary_RBS_proof}. For Latin squares that are the Cayley table of a group, the Hall-Paige conjecture \cite{hall_paige_conj} asked whether certain conditions were enough to guarantee that the Latin square contained a transversal, which is equivalent to there being a complete mapping of the group. This was confirmed in \cite{hall_paige_conj2, hall_paige_conj3}. 

Another useful concept for Latin squares is whether a Latin square can be decomposed into transversals, which is closely linked to orthogonality of Latin squares; see \cite{wanless_survey}. Due to such an important connection, a number of structures closely related to transversals are studied, including $k$-plexes \cite{trans_like1, trans_like2, trans_like3} and covers. Latin squares that contain varying lengths of maximal partial transversals have been studied \cite{many_maximal_partial_transversal}. There are also a variety of alternatives to regular orthogonality in Latin squares \cite{orth1, orth2, orth5, orth3, orth4}, and these may induce structures similar to transversals. For a recent review of the literature on transversals in Latin squares, see \cite{transveral_review}. 

While seemingly similar in definition, covers and lazy burning sets are distinct, with neither property implying the other. For example, in Figure \ref{fig:burning_example}, the lazy burning set $M$ provided is not a cover. However, there is a deeper relationship between the two; the \textit{complement} $V(H_L) \setminus M$ of $M$ is, in fact, a cover. Moreover, this set is a stronger form of a cover. 
For a Latin square $L$, define a \emph{cover-sequence} for $H_L$ as a sequence of entries $(e_1,\ldots,e_d)$ of $L$ such that 

\begin{enumerate}
    \item The set $\{e_1, \ldots, e_d\}$ is a cover of $H_L$.
    \item For each $i$, $e_i$ contains a row, a column, or a symbol not represented in $\{e_1, e_2, \ldots, e_{i-1}\}$, which is the set of earlier entries of the sequence.
\end{enumerate}

 We may view a cover-sequence as an online version of a minimal Latin square cover, noting that a minimal cover can be ordered to form a cover-sequence, but dropping the order of a cover-sequence does not, in general, result in a minimal cover.  We will show that these objects arise from the lazy burning sets of $H_L$, beginning with the following lemma. Observe that removing any vertex from a minimal lazy burning set results in a set that no longer successfully burns the hypergraph. Of course, every minimum lazy burning set is also minimal, but the reverse is not necessarily true. We denote the maximum length of a cover-sequence in $H_L$ as $\text{mcs}(L)$.

\begin{lemma}\label{lem:minimal_implies_cover}
Let $L$ be a Latin square of order $n$, and let $M$ be a minimal lazy burning set for $H_L$. We then have that $V(H_L)\setminus M$ contains at least one vertex in each row, column, and symbol class of $L$ and is, therefore, a cover of $L$.
\end{lemma}

\begin{proof}
Suppose for contradiction that this is not the case. 
We may then assume by symmetry that there is a row $r$ such that its line $\{(r,c',s') \in V(H_L)\}$ is a subset of $M$.  
Say vertex $(r,c,s)$ is one of these vertices, and note that if $M\setminus (r,c,s)$ was burned, then the vertex $(r,c,s)$ would be burned in the first propagation round since the hyperedge $\{(r,c',s') \in V(H_L)\}$ contains $n-1$ burned vertices. 
It follows that at least the vertices in $M$ will be burned after this one propagation round. 
Since $M$ was a lazy burning set, we then have that $M\setminus (r,c,s)$ is a lazy burning set, contradicting the minimality of $M$. 
This completes the proof. 
\end{proof}

As we have just shown, the complement of a minimal lazy burning set is always a cover, but not every cover can be ordered to form a cover-sequence; consider the trivial cover of every entry of $L$. However, we show below that the cover derived as the complement of a minimal lazy burning set can be ordered into a cover-sequence.

\begin{lemma}
\label{lem:lbs_implies_cover_seq} 
Let $L$ be a Latin square of order $n$, and let $M$ be a minimal lazy burning set for $H_L$. The vertices of $V(H_L)\setminus M$ can be ordered so that they form a cover-sequence for $L$.
\end{lemma}

\begin{proof}
Let $M$ be a minimal lazy burning set for $H_L$. 
Construct a sequence, $S$, from the vertices of $V(H_L)\setminus M$ so that the later a vertex is burned by propagation in the lazy burning process, the earlier it appears in the sequence, with ties ordered arbitrarily. 
We claim this sequence is a cover-sequence for $L$.

Suppose, by way of contradiction, that the constructed sequence $S$ is not a cover-sequence. By Lemma~\ref{lem:minimal_implies_cover}, we have that $V(H_L)\setminus M$ forms a cover for $L$. Thus, our sequence satisfies the first property of a cover-sequence, and hence it must fail the second property. Therefore, there is some vertex $v=(r,c,s)\in S$ such that row $r$, column $c$, and symbol $s$ are all represented by vertices earlier in the sequence. In particular, there are three vertices earlier in the sequence, say $x,y,z\in S$, each of which shares exactly one coordinate with $v$, say $x=(r,*,*)$, $y=(*,c,*)$, and $z=(*,*,s)$. 
Assume that $v$ was burned in round $t$ for some $t\geq 2$. By the definition of our ordering, vertices burned earlier must come later in the sequence, and so vertices $x$, $y$, and $z$ were not burned earlier than round $t$; in particular, $v$, $x$, $y$, and $z$ were not burned by the end of round $t-1$.

Consider the hyperedge that caused $v$ to be burned in round $t$. If $v$ burns due to the hyperedge corresponding to a given line, then at the end of round $t-1$, this hyperedge contained $n-1$ burned vertices. However, we know that both $v$ and exactly one of $x$, $y$, or $z$ belong to this hyperedge, and none of them were burned by the end of round $t-1$, so we have a contradiction. 
This completes the proof.
\end{proof}

\begin{figure}
\captionsetup[subfigure]{labelformat=empty}
    \null\hfill
    \subfloat[$S$]{\begin{tabular}{|c|c|c|c|c|}
    \hline
         \cellcolor{red!30}1 & 2 &3&4&5 \\ \hline
         2&4&1&5&3\\ \hline
         \cellcolor{red!30}3&\cellcolor{red!30}5&\cellcolor{red!30}4&\cellcolor{red!30}2&1 \\ \hline
         \cellcolor{red!30}4&\cellcolor{red!30}1&\cellcolor{red!30}5&\cellcolor{red!30}3&2 \\ \hline
         \cellcolor{red!30}5&\cellcolor{red!30}3&\cellcolor{red!30}2&\cellcolor{red!30}1&4 \\ \hline
    \end{tabular}}\hfill
    \subfloat[$M$]{\begin{tabular}{|c|c|c|c|c|}
    \hline
         \cellcolor{red!30}1 & 2 &\cellcolor{red!30}3&4&5 \\ \hline
         2&4&1&5&3\\ \hline
         \cellcolor{red!30}3&\cellcolor{red!30}5&\cellcolor{red!30}4&\cellcolor{red!30}2&1 \\ \hline
         \cellcolor{red!30}4&\cellcolor{red!30}1&\cellcolor{red!30}5&\cellcolor{red!30}3&2 \\ \hline
         \cellcolor{red!30}5&\cellcolor{red!30}3&\cellcolor{red!30}2&\cellcolor{red!30}1&4 \\ \hline
    \end{tabular}\hfill}\hfill
    \subfloat[$L \setminus M$]{\begin{tabular}{|c|c|c|c|c|}
    \hline
    & \cellcolor{blue!25}4& & \cellcolor{blue!25}1& \cellcolor{blue!25}5\\ \hline
    \cellcolor{blue!25}6& \cellcolor{blue!25}3& \cellcolor{blue!25}7& \cellcolor{blue!25}2&\cellcolor{blue!25}8 \\ \hline
    & & & &\cellcolor{blue!25}9 \\ \hline
    & & & &\cellcolor{blue!25}10\\ \hline
    & & & &\cellcolor{blue!25}11\\ 
    \hline
    \end{tabular}}
    \hfill\null
    \caption{A 5$\times 5$ Latin square with a minimum lazy burning set $S$, a non-minimum lazy burning set $M$, and a cover-sequence formed from $L \setminus M$, with the entries in blue numbered in their order in the sequence.}
\label{fig:cov_seq_example}
\end{figure}

We note that the minimality of $M$ in Lemma~\ref{lem:lbs_implies_cover_seq} is necessary; for example, consider the set $M=V(H_L)\setminus\{v\}$ for some vertex $v$. The set $M$ is a non-minimal lazy burning set for $H_L$; moreover, the set $V(H_L)\setminus M$ is not a cover for $L$ and hence cannot be ordered to form a cover-sequence.

Additionally, the converse of Lemma~\ref{lem:lbs_implies_cover_seq} is not necessarily true. 
In particular, there exist Latin squares $L$ and non-minimal lazy burning sets $M$ for $V(H_L)$ such that $V(H_L)\setminus M$ can be ordered to form a cover-sequence; see Figure~\ref{fig:cov_seq_example}. However, a partial converse to Lemma~\ref{lem:lbs_implies_cover_seq} holds, for which we introduce the following definition. Define a \emph{weak cover-sequence} of $L$ as a sequence of entries of $L$, $(e_1, e_2,\ldots, e_d)$, such that, for each $i$, $e_i$ contains a row, a column, or a symbol not represented in $\{e_1, e_2, \ldots, e_{i-1}\}$ (so $\{e_1, e_2,\ldots, e_d\}$ may or may not form a cover of $L$). We then see that any weak cover-sequence has a corresponding lazy burning set, which need not be minimal.

\begin{lemma}
\label{lem:cover_seq_implies_lbs} 
Let $L$ be a Latin square of order $n$. Choose $M\subseteq V(H_L)$ so that there exists an ordering of $V(H_L)\setminus M$ that is a weak cover-sequence. We then have $M$ is a lazy burning set for $H_L$.
\end{lemma}

\begin{proof}
Let $ u_1, u_2, \ldots, u_k$ be an ordering of $ V (H_L) \setminus M$ that gives a weak cover-sequence for $L$. We will show by induction that when $M$ is burned, each vertex $u_i$ will burn through propagation. 

The induction will occur in reverse order, so the base case is that $u_k$ will burn. The vertices that are not burning are exactly those in $\{u_1,u_2,\ldots,u_k\}$, and $u_k$ has at least one coordinate that is not shared by any of the $u_i$ with $i < k$. Suppose the \emph{row} coordinate of $u_k$ is not shared by any of the other $u_i$. No vertex $u_i$ with $i<k$ belongs to this row, which is to say that no unburned vertex belongs to this row except for $u_k$. Hence, $u_k$ burns due to a row hyperedge. The cases for column and symbol hyperedges are similar.

Assume for some $\ell\geq 0$ that $u_k,u_{k-1},\ldots,u_{k-\ell}$ have burned. We must prove that $u_{k-\ell-1}$ burns. First, observe that the only vertices that could be unburned at this point are $u_1,u_2,\ldots,u_{k-\ell-1}$. Also, $u_{k-\ell-1}$ has at least one coordinate that is not shared by any $u_i$ with $i<k-\ell-1$. Suppose it is the row coordinate. It follows that $u_{k-\ell-1}$ belongs to a row that does not contain any of $u_1,u_2,\ldots,u_{k-\ell-2}$. That is, all other vertices in this row either belong to $M$ or $\{u_{k-\ell}, u_{k-\ell+1},\ldots,u_k\}$. Since these vertices are burned, the row causes $u_{k-\ell-1}$ to burn. The column and symbol cases are similar.
\end{proof}

We thus obtain the following equivalence as a combination of Lemmas~\ref{lem:lbs_implies_cover_seq} and \ref{lem:cover_seq_implies_lbs}.

\begin{theorem}
\label{lbs_weak_characterization}
If $L$ is a Latin square of order $n$ and $M\subseteq V(H_L)$, then $M$ is a lazy burning set for $H_L$ if and only if the vertices of $V(H_L)\setminus M$ can be ordered so that they form a weak cover-sequence for $L$. 
\end{theorem}

\begin{proof}
The reverse direction is exactly Lemma~\ref{lem:cover_seq_implies_lbs}. For the forward direction, $V(H_L)\setminus M$ may or may not form a cover for $L$ since $M$ may or may not be minimal. The same argument as in the proof of Lemma~\ref{lem:lbs_implies_cover_seq} works here to show that $V(H_L)\setminus M$ can be ordered to form a weak cover-sequence, satisfying the second property of a cover-sequence.
\end{proof}

In particular, we note that if $L$ is a Latin square of order $n$ and $M$ is a minimum lazy burning set for $H_L$, then $|V(H_L)\setminus M|=\mathrm{mcs}(L)$. The following theorem characterizes the lazy burning number of $H_L$ in terms of $\mathrm{mcs}(L).$

\begin{theorem}
\label{thm:lazyburning_mcs_equality} If $L$ is a Latin square of order $n$, then $b_L(H_L) = n^2 - \mathrm{mcs}(L)$.
\end{theorem}
Therefore, bounds on $\mcs(L)$ will give corresponding bounds on $b_L(H_L)$, transforming the lazy burning of a Latin square into one of finding cover-sequences. However, determining the exact value of $\mcs(L)$ can be somewhat unwieldy. We approach the idea of finding maximal cover-sequences from a new angle, introducing a novel variant of the \emph{chain} of subsquares of a Latin square $L$, which is a collection of subsquares of $L$, say $S_0, S_1,\ldots, S_\alpha$, such that $\emptyset = S_0 \subseteq S_1 \subseteq \cdots \subseteq S_\alpha = L$, where $S_1$ is the trivial $1\times 1$ square consisting of a single cell. 

Although chains of Latin squares have not yet been well studied in general, the study of subsquares in Latin squares is a topic of considerable interest. The expected number of subsquares in a random Latin square were studied in \cite{subsq_in_rectangles, canonical_labellings,MR4488316,Kwan_substructures, MR1685535}, and include important connections to canonical labelings of Latin squares in \cite{canonical_labellings}. Latin squares with disjoint subsquares of prescribed sizes have been studied \cite{realization_five, realization_hypercubes}, as have Latin squares where each entry occurs in a proper subsquare \cite{Valle_Dukes}.
Subgroups of groups correspond to subsquares of the corresponding Caley Table of the Latin square, and many interesting connections exist here; see \cite{many_maximal_partial_transversal}.  

We call a chain \emph{connected} if for each $i$ with $1 \leq i \leq \alpha$, there is a line $\ell$ in $S_{i}$ but not in $S_{i-1}$ so that $S_{i}$ is the smallest subsquare intersecting $\ell$ and containing all the entries of $S_{i-1}$. The \emph{length} of a chain is one less than the chain's cardinality, and we denote the length of the smallest connected chain as $\text{scc}(L)$. This quantity is related to $\mcs(L)$, as we show below.

\begin{theorem} \label{thm:mcs_equality_FuncOf_stc}
For a Latin square $L$ of order $n$, $\mathrm{mcs}(L) = 3n - 1 - \mathrm{scc}(L)$.
\end{theorem}

We prove Theorem~\ref{thm:mcs_equality_FuncOf_stc} in two lemmas. Additionally, we first define a useful property of elements of a weak cover-sequence:  the \emph{weight} of an entry $e_i$ in a weak cover-sequence is the number of lines $\ell$ such that $e_i\in \ell$ and $e_j \cap \ell = \emptyset$ for all $j< i$. Informally, it is the number of unique lines that $e_i$ adds to the weak cover-sequence, and thus, for a proper cover-sequence, we see that the weight of all of its entries must sum to exactly $3n$.

We begin with the lower bound of Theorem~\ref{thm:mcs_equality_FuncOf_stc}.

\begin{lemma} \label{lem:mcs_bounded_below_stc}
For a Latin square $L$ of order $n$, $\mathrm{mcs}(L) \geq 3n - 1 - \mathrm{scc}(L)$.
\end{lemma}

\begin{proof}
Let $S_0,S_1, \ldots, S_\alpha$ be a smallest connected chain of $L$. Let $e_1$ be the unique entry in $S_1$. For each $i$ with $2 \leq i \leq \alpha$, let $\ell_i$ be a line  in $S_{i}$ but not in $S_{i-1}$ so that $S_{i}$ is the smallest subsquare intersecting $\ell_i$ and containing all the entries of $S_{i-1}$. Define $e_i \in S_i\setminus S_{i-1}$ such that $e_i$ is in line $\ell_i$ and is also in one other line that intersects $S_{i-1}$, and define $F_i$ to be a maximal sequence of vertices $(f_i^1, f_i^2, \ldots)$ such that $f_i^j$ has exactly two lines that intersect $\{f_i^1, f_i^2,\ldots, f_i^{j-1}\}\cup \bigcup_{i'<i} \left( \{e_{i'}\} \cup F_{i'}\right)$. 

We claim that $(e_1) F_1 (e_2) \cdots (e_i) F_i$ is a cover-sequence of $S_i$, which only requires us to show that the set of its vertices is a cover. Let $R,C,S$ be the rows, columns, and symbols in the proposed cover-sequence of $S_i$. The symbols in rows $R$ and columns $C$ must each be in $S$, or else $F_i$ could be extended, and so would not be maximal. Similarly, the rows in columns $C$ that contain a symbol in $S$ must be rows in $R$, and similarly, the columns in rows $R$ that contain a symbol in $S$ must be columns in $C$. These $R$, $C$, and $S$, therefore, define exactly the rows, columns, and symbols of a subsquare that intersects line $\ell_i$ and, by induction, contains $S_{i-1}$ as a subsquare. Since $S_i$ is the smallest subsquare of $L$ that satisfies this condition, each row, column, and symbol of $S_i$ is represented in the sequence $(e_1) F_1 (e_2) \cdots (e_i) F_i$. This sequence only contains entries in $S_i$, so it must be a cover-sequence for $S_i$. Using this inductive step until $i=\alpha$, it then follows that $(e_1) F_1 (e_2) \cdots (e_\alpha) F_\alpha$ must be a cover-sequence for $S_\alpha$. 
 
We may now treat the sequence constructed as a cover-sequence. We have that $e_1$ has weight $3$. For $2 \leq i\leq \alpha$, by definition, $e_i$ shares exactly one line with $S_{i-1}$. The subsequence $(e_1) F_1 (e_2) \cdots (e_{i-1}) F_{i-1}$ intersects each line in $S_{i-1}$. It follows that $e_i$ shares exactly one line with earlier vertices in the sequence. Hence, the weight of $e_i$ is exactly $2$. All vertices in the sequences $F_i$ share two lines with earlier vertices in the sequence and so have weight $1$. 

The sum of the weights must equal $3n$ since this is the number of lines that must be covered, and each line contributes a value of $1$ to the weight of exactly one vertex. We have one vertex of weight $3$, a further $\alpha-1$ vertices of weight $2$, which are the vertices $e_2, e_3, \ldots, e_\alpha$, and the remaining $\beta$ vertices have weight $1$. We then find that $3n = 3 + 2 (\alpha-1) + \beta$. The length of the sequence is $1 + (\alpha-1) + \beta$, and subbing in $\beta=3n-3-2(\alpha-1)$, we find this length is $1+(\alpha-1)+(3n-3-2(\alpha-1)) = 3n - \alpha -1$. Since $\alpha=\mathrm{scc}(L)$, it follows that the length of the sequence is $3n-1-\mathrm{scc}(L)$. Since $\text{mcs}(L)$ is the maximum size of such a sequence, it follows that $\text{mcs}(L) \geq 3n - 1 - \text{scc}(L)$, which completes the proof. 
\end{proof}

We next derive the upper bound of Theorem~\ref{thm:mcs_equality_FuncOf_stc}.

\begin{lemma} \label{lem:mcs_bounded_above_stc}
For a Latin square $L$ of order $n$, $\mathrm{mcs}(L) \leq 3n - 1 - \mathrm{scc}(L)$.
\end{lemma}
\begin{proof}
Let $E=(v_1, v_2,\ldots, v_d)$ be a cover-sequence of length $d = \mathrm{mcs}(L)$. 
By making small changes to $E,$ we will construct a cover sequence with a length of $d$ and several nice properties. 
We will then use this modified cover-sequence to construct a connected chain of length $3n-1-d$.
Since this connected chain must have length at least $\mathrm{scc}(L)$ and the modified cover-sequence has length at most $\mathrm{mcs}(L)$, we have that $scc(L) \leq 3n-1-d = 3n-1-mcs(L)$, and the result follows. 

By permuting the entries of $L$, we may assume that $v_1=(1,1,1)$, which has weight $3$ in the original cover-sequence. 
If some other entry $v_i = (r,c,s)$ of the cover-sequence had weight $3$, then we could insert the entry $(1,c,s')$ into the sequence, directly after $v_1$, where $s' \notin \{s, 1\}$. 
This new entry in the sequence must have a weight of $2$, and its insertion would reduce the weight of $v_i$ from $3$ to $2$. 
There must also be a unique entry $v_j$ in the sequence that represents the symbol line for symbol $s'$. 
The weight of $v_j$ will drop by $1$ by the addition of $(1,c,s')$. 
If the weight of $v_j$ drops from $1$ to $0$, we must remove $v_j$ from the sequence to ensure the sequence is still a cover-sequence. 
Whether $v_j$ was removed from the sequence or not, we are left with a sequence that is not shorter than the original sequence but with one less vertex of weight $3$.
This argument shows that all entries of weight $3$ can be eliminated in this fashion except for one, which is the first vertex in the cover-sequence.

We may then assume that the only entry of the cover-sequence $(v_1, v_2,\ldots, v_d)$ of weight $3$ is the first element. 
We define $e_1=v_1$ and $L_1=\{v_1\}$. 
Recursively, let $e_i$ be the first entry of the cover-sequence that is not in the subsquare $L_{i-1}$ and define $L_i$ as the smallest subsquare of $L$ that contains the entries of $L_{i-1}$ as well as the entry $e_i$. 
Define $R_i,C_i$, and $S_i$ as the rows, columns, and symbols in $L_i$. 

Suppose that there is a row $r$ and column $c$ that is represented in the cover-sequence before $e_i$ but where the symbol $s$ with $(r,c,s)\in L$ is represented in an entry in the cover-sequence at or after $e_i$, say in the sequence element $v_j$. 
If $v_j = (r,c,s)$, then we may move it earlier in the sequence so that it occurs before $e_i$.
Otherwise, we can add entry $(r,c,s)$ to the cover-sequence before $e_i$ but after the sequence elements that represent $r$ and $c$. 
In both cases, the weight of the cover-sequence element $(r,c,s)$ is $1$. 
In the latter case, it might be that $v_j$ now has weight $0$, in which case we remove it from the cover sequence. This operation with $(r,c,s)$ does not decrease the length of the cover sequence.
We may recursively repeat this procedure and a similar process, but in which only the column or row is represented at or after $e_i$. 
This procedure only terminates when for each row $r$ and column $c$ represented before $e_i$, the corresponding symbol in row $r$ and column $c$ of the Latin square is also represented before $e_i$, and similarly when the roles of row, columns, and symbols are switched. 
However, this is exactly the Latin property, and so the set of lines represented before $e_i$ forms a subsquare, namely $L_i$. 

Consider one of the lines that intersects $e_i$ that did not intersect $L_{i-1}$, say $\ell_i$. 
The smallest subsquare, say $L'$, that contains $L_{i-1}$ and intersects $\ell_i$ will contain the entry $e_i$, and so would contain $L_i$ as a subsquare. 
However, $L_i$ contains $L_{i-1}$ and intersects $\ell_i$, and so contains $L'$ as a subsquare, implying $L'=L_i$. 
It then follows that $L_0, L_1, \ldots, L_\alpha$ is a connected chain. 

In this constructed sequence of length $d$, the first sequence element has weight $3$, the first sequence element chosen in each $L_i$ for $2 \leq i \leq \alpha$ has weight $2$, and the $d-\alpha$ remaining sequence elements have weight $1$. 
Since the sum of the weights is $3n$, we have $3n = 3(1) + 2(\alpha-1) + 1(d-\alpha) = \alpha +d +1$, from which it follows that $\alpha = 3n-1-d$. 
Thus, the length of the chain that we constructed is $3n-1-d$, as required. \end{proof}

Combining Theorems~\ref{thm:lazyburning_mcs_equality} and \ref{thm:mcs_equality_FuncOf_stc} now provides an exact value for $b_L(H_L)$ in terms of $\mathrm{scc}(L)$. 

\begin{corollary}
\label{exact_value_H_L}
If $L$ is a Latin square of order $n$, then $b_L(H_L) = n^2 - 3n +1 +\mathrm{scc}(L)$.
\end{corollary}

It is straightforward to bound $\mathrm{scc}(L)$ once we know that the order of a subsquare can be at most half the size of the order of the Latin square it embeds into; see \cite{designs_handbook}. 

\begin{theorem} \label{thm:stc_bounds}
If $L$ is a Latin square of order $n$, then $2 \leq \mathrm{scc}(L) \leq \lfloor \log_2{n} \rfloor+1$.
\end{theorem}

\begin{proof}
For the lower bound, any chain $S_0\subseteq S_1\subseteq \cdots\subseteq S_\alpha$ for $L$ must contain $S_0=\emptyset$, the trivial $1\times 1$ square $S_1$, and $S_\alpha=S$. Hence, any chain includes at least three distinct subsquares and, therefore, has a length of at least two.

For the upper bound, we let $|\cdot|$ represent the order of a Latin square and note that any chain $S_0\subseteq S_1\subseteq \cdots\subseteq S_\alpha$ of $L$ must satisfy $|S_{i+1}|\geq 2|S_{i}|$ for $i \geq 1$.
For induction, we have $|S_i| \geq 2^{i-1}$  and $|S_1|=1$, from which it follows that $|S_\alpha|\geq 2^{\alpha-1}$. 
Since $|S_\alpha|\leq n$, it follows that $2^{\alpha-1}\leq n$ and so $\alpha \leq (\log_2{n})+1$, and the proof is complete. 
\end{proof}

Combining this with previous results, we arrive at our main result, giving asymptotically tight bounds on both $\text{mcs}(L)$ and $b_L(H_L)$. 

\begin{corollary} \label{cor:bounds_Bl}
If $L$ is a Latin square of order $n$, then the following statements hold:
\begin{enumerate}
\item $n^2 -3n +3 \leq b_L(H_L) \leq n^2-3n + 2 + \lfloor \log_2{n} \rfloor .$ 
\item $3n - 2 - \lfloor \log_2{n} \rfloor \leq \mathrm{mcs}(L) \leq 3n -3.$
\end{enumerate}
\end{corollary}

The bounds in Corollary~\ref{cor:bounds_Bl} (1) on $b_L(H_L)$ are both achieved by well-known Latin squares, and we present both of those results below, beginning with the lower bound.

\begin{theorem}\label{lem:cyclic_scc}
If $n \geq 1$ and $L$ is the cyclic Latin square of size $n$, then $\mathrm{scc}(L)=2.$ In particular, $b_L(H_L)=n^2 -3n +3.$
\end{theorem}

\begin{proof}
To show that $\mathrm{scc}(L)=2$, it suffices to show that $L$ is the smallest subsquare containing some $1 \times 1$ cell $S_1$ and some line $\ell$ that does not intersect $S_1$. Let $S_1$ be the subsquare formed by the cell in column 1, row 1, containing the symbol 1, and let $\ell$ be the line representing the symbol 2. We note that as $L$ is the cyclic Latin square, it has a direct correspondence with the addition table for the group $\mathbb{Z}_n$; furthermore, any subsquare of $L$ has a direct correspondence with some coset of $\mathbb{Z}_n$; see \cite[Theorem~1.6.4]{keed}. Thus, any subsquare that must contain $S_1$ and intersect $\ell$ corresponds to a coset of $\mathbb{Z}_n$ that contains two consecutive integers. This can only be all of $\mathbb{Z}_n$, implying that $L$ is the smallest such subsquare and showing that $L$ contains a connected chain of length 2, proving the claim.
\end{proof}

Another such example that shows the lower bound exists. 
The existence of Latin squares without subsquares of certain orders has been extensively studied, and recently a preprint has claimed to settle the existence of proper subsquare-free Latin squares for all orders \cite{no_subsquares}. 
In particular, in these proper subsquare-free Latin squares, no chain has length larger than $2$.

\begin{lemma}\cite{no_subsquares}
For $n\notin \{4,6\}$, there exists a Latin square $L$ of order $n$ without proper subsquares, and so $scc(L)=2$. 
\end{lemma}

To show the tightness of the upper bound of Corollary~\ref{cor:bounds_Bl} (1), we use powers of intercalates. Given two Latin squares $L_1 \subseteq R_1 \times C_1 \times S_1$ and $L_2 \subseteq R_2 \times C_2 \times S_2$, their \textit{product} is the Latin square $L=L_1 \times L_2 \subseteq (R_1 \times R_2) \times (C_1 \times C_2) \times (S_1 \times S_2)$, which is formed by including entry $((r_1,r_2),(c_1,c_2),(s_1,s_2))$ for each pair $(r_1, c_1, e_1)\in L_1$ and $(r_2, c_2, e_2)\in L_2$. We additionally define $\mathcal{I}$ as the \emph{intercalate}, which is the unique $2 \times 2$ Latin square, and define $\mathcal{I}^k$ as the product $\mathcal{I}^{k-1} \times \mathcal{I}$. 

\begin{theorem} \label{thm:stc_mod_Z2k} If $\mathcal{I}$ denotes the intercalate and $k \geq 1$, then $\mathrm{scc}(\mathcal{I}^k)=k+1$.
\end{theorem}
\begin{proof}
Theorem~\ref{thm:stc_bounds} shows the upper bound. 
We first note that the Latin square $\mathcal{I}^k$ is equivalent to the addition or Cayley table of the group $(\mathbb{Z}_2^k,+)$, where $+$ represents the bitwise XOR operation on the binary words $\mathbb{Z}_2^k$. 
As such, if $S$ is a subsquare of $\mathcal{I}^k$, then so is $\{(r+v, c+v, s) : (r,c,s) \in S\}$ and so is $\{(r, c+v, s+v) : (r,c,s) \in S\}$, for any $v \in \mathbb{Z}_2^k$. 
This can be seen in the first case by noting that $(r+v) + (c+v) = (r+c)+2v = r+c = s$ in the group $(\mathbb{Z}_2^k,+)$. 
As a consequence of this, any subsquare can be mapped to a subsquare of the same size with $((0,0,\ldots, 0), (0,0,\ldots, 0), (0,0,\ldots, 0))$ in the subsquare, and so we may assume that there is some chain that contains this entry in all of the chain's subsquares.  

Consider some subsquare of the chain. 
If $y$ is a row that occurs in that subsquare, then since $y + (0,0,\ldots, 0) = y$, it follows that symbol $y$ must occur in the subsquare in row $y$ and column $(0, 0\ldots, 0)$. 
Similarly, $y$ will be a column represented in the subsquare since symbol $y$ is in row $(0, 0,\ldots, 0)$ in column $y$. 
It follows that the subsquare is a subset of $Y\times Y \times Y$ for some subset of elements $Y$ from the group $(\mathbb{Z}_2^k,+)$. 
As such, if $y_1,y_2 \in Y$, it follows that row $y_1$ and column $y_2$ will contain symbol $y_1+y_2$, and so we must have $y_1+y_2\in Y$. 
We also have $y+y=(0,0,\ldots, 0)$ for all $y$.
This implies that $Y$ has closure under addition and inverse, and so is a subgroup of $(\mathbb{Z}_2^k,+)$. 

Suppose for the sake of contradiction that there is a shortest connected chain $S_0 \subseteq \cdots \subseteq S_\alpha$ with $\alpha < k+1$. 
This gives that there is some $i$ with $2|S_i|<|S_{i+1}|$. 
Suppose $R_i$ is the corresponding subgroup of $(\mathbb{Z}_2^k,+)$ coming from $S_i$, and $R_{i+1}$ defined similarly. 
If $x$ is a group element in $R_{i+1}$ but not $R_{i}$, then consider the smallest subgroup that contains both $R_i$ and $x$, say $R_i^x$. 
This subgroup $R_i^x$ must be a subset of or equal to $R_{i+1}$. 
We have that $r,r+x \in R_i^x$ for each $r \in R_i$. 
We also have that $r+x \neq r'$ and $r+x \neq r'+x$ for $r, r' \in R_i$. 
It follows that $R'=R_i \cup \{x+r : r \in R_i\}$ has cardinality $2|R_i|$. 
Since $(x+r_1)+(x+r_2) = r_1+r_2 \in R_i\subseteq R^\prime$ and  $(x+r_1) + r_2 = x + (r_1+r_2) \in R' \setminus R_i$, it follows that $R'$ is closed under addition. 
As $R'$ is also closed under inversion, we have that $R'$ is a subgroup, and so it must be that $R' = R_i^x$. 
The entries of the Latin square $\mathcal{I}^k$ in rows $R'$ and columns $R'$ contain the symbols $R'$, and so these entries form a subsquare that strictly contains $S_i$ and is strictly contained in $S_{i+1}$. 
This argument holds for all $x$ and, therefore, violates the chain's connectedness, forming the required contradiction and completing the proof.  
\end{proof}

Together, Theorems~\ref{lem:cyclic_scc} and \ref{thm:stc_mod_Z2k} show that the bounds of Theorem~\ref{thm:stc_bounds} and Corollary~\ref{cor:bounds_Bl} are tight. We have shown that the lazy burning number of an $n$-uniform Latin square hypergraph is determined entirely by its shortest connected chain, and the length of this chain is at most $\log_2 n +1$. We explore this relationship between the lazy burning number and the shortest connected chain more in the next section, where we consider the separate but similar problem of lazily burning the $3$-uniform construction of a Latin square hypergraph.

\section{Burning the 3-uniform case}

The 3-uniform construction forces a large amount of structure onto any lazy burning set; a nontrivial set of burned rows, columns, and symbols will ``fill out'' to a subsquare of $L$. 
We note that the lazy burning number of $3$-uniform linear hypergraphs has previously been studied \cite{expander_triple_systems}, although using differing terminology. 

\begin{lemma}\label{lem:3unisubsquare}
Let $L$ be a Latin square, let $H^L = (R,C,S)$ be its associated hypergraph, and let $R^*$, $C^*$, and $S^*$ respectively denote a set of rows, columns, and symbols that are burned. If at least two of $r^*=|R^*|$, $c^*=|C^*|$, and $s^*=|S^*|$ are nonzero, then after the burning finishes propagating, we have that the resulting set $M$ of burned lines corresponds to the smallest subsquare of $L$ that intersects $R^* \cup C^* \cup S^*$. 
\end{lemma}
\begin{proof}
    Without loss of generality, let $r^*,c^*\geq 1$. Define $r^\prime$, $c^\prime$, and $s^\prime$ as the number of rows, columns, and symbols, respectively, that are burned when propagation stops. We then have that $r^\prime,c^\prime\geq1$. As $L$ is a Latin square, in every $r^\prime\times c^\prime$ sized subsection of $L$, there must be at least $\max(r^\prime,c^\prime)$ different symbols. Thus, when the propagation stops, at least $\max(r^\prime,c^\prime)$ different symbols are burned (either due to their intersection with the lazy burning set or because they burned due to a burned pair $(r,c)$). Therefore, $s^\prime \geq \max(r^\prime,c^\prime).$ A similar argument can then be used to show that $r^\prime \geq \max(c^\prime,s^\prime)$ and that $c^\prime \geq \max(r^\prime,s^\prime)$, forcing $r^\prime = c^\prime = s^\prime$. The set $M$ has the same number of rows, columns, and symbols and is thus a subsquare of $L$ that contains $R^* \cup C^* \cup S^*$. 
    
    Suppose there is a smaller subsquare $N$ that intersects $R^* \cup C^* \cup S^*$. Let $w$ be the first line in $M\setminus N$ to burn through propagation (if multiple lines in $M\setminus N$ burned simultaneously, choose one). Without loss of generality, assume $w$ is a symbol line. It burned due to a burned row line $u$ and a burned column line $v$, both intersecting $N$. We then have the triple $(u,v,w)$ is an entry that belongs to $N$, so $w$ is a symbol intersecting $N$, which is a contradiction. This completes the proof.
\end{proof}

This observation provides a powerful tool for analyzing the lazy burning of the 3-uniform construction. It recalls the notion of a connected chain, as it seems that the propagation resulting from a lazy burning set $B$ in this construction corresponds nearly exactly to filling out surrounding subsquares that contain $B$ and intersect a new line of $L$. The shortest connected chain should then provide a blueprint for how to lazily burn $H^L$, which is the case. 

\begin{theorem}\label{thm:3unistcupper}
    If $L$ is a Latin square, then $b_L(H^L) = \mathrm{scc}(L)+1$.
\end{theorem}

\begin{proof}
We will show that $b_L(H^L)$ is bounded both above and below by $\mathrm{scc}(L)+1$. For the lower bound, let $B = \{e_0,e_1,\ldots,e_k\}$ be a set of lines of $L$ that form a minimum lazy burning set for $H^L$, so $b_L(H^L)=k+1$. Without loss of generality, assume that the lines $e_0$ and $e_1$ are of different types (for example, a row and a column). Let $S_i$ represent the smallest subsquare of $L$ containing $\{e_0,e_1,\ldots,e_i\}$. Because $B$ is a lazy burning set, by Lemma~\ref{lem:3unisubsquare}, we must have that $S_k=L$, and thus by construction, it follows that $\emptyset = S_0 \subseteq S_1 \subseteq \ldots \subseteq S_k=L$ is a connected chain. Therefore, we have that $k \geq \mathrm{scc}(L)$, implying that $b_L(H^L) \geq \mathrm{scc}(L)+1$, which yields the lower bound.

For the upper bound, let $\emptyset = T_0 \subseteq T_1 \subseteq \ldots \subseteq T_j = L$ be a shortest connected chain of $L$, and let $\ell_0, \ell_1$ be two lines of $L$ whose intersection forms $T_1$. By definition, for each $1 \leq i \leq j-1$, there exists a line $\ell_{i+1}$ of $L$ such that $T_{i+1}$ is the smallest subsquare containing $T_i$ that intersects $\ell_{i+1}$. We claim that $C = \{\ell_0,\ell_1,\ldots,\ell_j\}$ is a lazy burning set for $L$, and indeed it must be, as the connectedness of the $T_i$ implies that $L$ is the only subsquare that contains all of $C$. Therefore, $b_L(H^L) \leq |C|=j+1=\mathrm{scc}(L)+1$, finishing the proof.
\end{proof}

By combining Theorem~\ref{thm:3unistcupper} and Corollary \ref{exact_value_H_L}, we immediately obtain the following result, which allows us to infer $b_L(H_L)$ given $b_L(H^L)$ and vice-versa.

\begin{theorem}
\label{relate_both_constructions}
If $L$ is an $n\times n$ Latin square, then $b_L(H_L)-n^2=b_L(H^L)-3n$.
\end{theorem}

We thus find that in both the $n$-uniform and $3$-uniform constructions of Latin square hypergraphs, determining the lazy burning number is equivalent to determining the shortest connected chain of the Latin square and the maximum cover-sequence. This suggests that both the shortest connected chain and the maximum cover-sequence might be more fundamental properties of a Latin square. In the $3$-uniform case, $\mathrm{scc}(L)$ measures exactly how many unique lines are needed to ensure that $L$ itself is the only subsquare that contains them all. We find it natural to relate this behavior with those of basis vectors in a vector space and thus identify $\mathrm{scc}(L)$ with the notion of \textit{dimension}; that is, how many unique elements are needed to span the entire space?

For finitely generated groups $G$, we find that the number of unique lines needed to span $L(G)$ is exactly two more than the number of generators of $G$, which we present in Theorem~\ref{thm:scc_group}. 
With this in mind, the lazy burning number still seems to be a natural form of dimension for a Latin square, as it adheres closely to the definition of dimension for other combinatorial objects, such as the dimension of a Steiner triple system, which we discuss in more detail in Section~4.

It is known that the columns and rows of any subsquare of a Latin square $L(G)$ generated by some group $(G,\oplus)$ must respectively be some coset $H\oplus a$ and $b \oplus H$ for suitably chosen $H,$ $a$, and $b$; see \cite[Theorem 1.6.4]{keed}. We proceed first by proving a technical lemma on the composition of such cosets. Note that lines $\ell$ of a Latin square each correspond to particular group elements, and two lines are considered distinct if they are lines of different type (for example, column and row) even if they both correspond to the same group element. Given a group $G$, we will let $\langle g_1,g_2,\ldots,g_n\rangle$ be the subgroup generated by the elements $g_1,g_2,\ldots,g_n$. We let $g^{-1}$ denote the group inverse of $g$.

\begin{lemma}\label{lem:coset_square}
    Let $(G,\oplus)$ be a group, let $L(G)$ be the Cayley table of $G$ with the borders removed, and let $\{\ell_1, \ell_2,\ldots,\ell_k\}$ be distinct lines of $L(G)$. Fix a row $r$ and a column $c$, and define the columns $c_i^*$ as follows: 
    $$c_i^* =\begin{cases} \ell_i & \ell_i \text{ is a column,} \\ r^{-1}\oplus \ell_i & \ell_i \text{ is a symbol,} \\ r^{-1} \oplus \ell_i \oplus c & \ell_i \text{ is a row.} \end{cases}$$ 
Define $s_i = c_i^* \oplus c^{-1}$. If $S$ is the smallest subsquare of $L(G)$ that intersects the lines $\{r,c,\ell_1,\ldots,\ell_k\}$, and $\mathrm{Col}(S)$ and $\mathrm{Row}(S)$ are respectively the set of columns and rows that $S$ intersects, then 
    \begin{align*}
        \mathrm{Col}(S) &= \langle s_1,\ldots,s_k\rangle \oplus c , \text{ and}\\
        \mathrm{Row}(S) &= r \oplus \langle s_1,\ldots,s_k \rangle.
    \end{align*}
\end{lemma}
\begin{proof}
    Using \cite[Theorem 1.6.4]{keed}, it is sufficient to show that this claim holds only for $\mathrm{Col}(S).$ We show that $\mathrm{Col}(S)$ is both a subset and superset of $\langle s_1,\ldots,s_k\rangle \oplus c$, beginning with the subset argument. By Lemma~\ref{lem:3unisubsquare}, we have that $S$ is the same subsquare that results from burning the lines $\{r,c,\ell_1,\ldots,\ell_k\}$. Consider now $\mathrm{Col}(S)$ and the line $\ell_i$. If $\ell_i$ is a column, then $\mathrm{Col}(S)$ must intersect $\ell_i$. If $\ell_i$ is a symbol, then because row $r$ is burned, we have that $\mathrm{Col}(S)$ must intersect $r^{-1}\oplus \ell_i$. 
    Finally, if $\ell_i$ is a row, then because column $c$ is burned, the symbol $\ell_i \oplus c$ must be burned, and so due to row $r$, we have that $\mathrm{Col}(S)$ must intersect the column $r^{-1} \oplus \ell_i \oplus c$. 
    
    Therefore, the lines $\{r,c,c_1^*,\ldots,c_k^*\}$ must all intersect $S$, and moreover, we have that $S$ must be the smallest subsquare intersecting these lines.  To see this, note that by a similar argument performed above, any subsquare that intersects the lines $\{r,c,c_1^*,\ldots,c_k^*\}$ must also intersect the lines $\{r,c,\ell_1,\ldots,\ell_k\},$ and vice versa. 
    Thus, as $\langle s_1,s_2,\ldots,s_k\rangle \oplus c$ intersects the lines $\{c,c_1^*,\ldots,c_k^*\}$, due to the minimality of $S$ we must have that $\mathrm{Col}(S) \subseteq \langle s_1,\ldots,s_k\rangle \oplus c.$
    
    We show the reverse inclusion as follows. If $g \in \langle s_1,s_2,\ldots,s_k\rangle \oplus c$, then there there exists some positive integers $\alpha_i$ and natural number $n$ such that $g = s_{\beta_1}^{\alpha_{1}}\oplus s_{\beta_2}^{\alpha_{2}}\oplus\cdots \oplus s_{\beta_n}^{\alpha_n} \oplus c,$ where each $\beta_i \in \{1,2,\ldots,k\}$. We wish to show that all elements of this form belong to $\mathrm{Col}(S)$. By Lemma~\ref{lem:3unisubsquare}, it suffices to show that columns of this form become burned after burning the lines $\{r,c,\ell_1,\ldots,\ell_k\}$. We will therefore show that if the lines $\{r,c,\ell_1,\ldots,\ell_k\}$ are burned, then given any burned column $h'$ of the form $s_{\beta_1}^{\alpha_{1}}\oplus s_{\beta_2}^{\alpha_{2}}\oplus\cdots \oplus s_{\beta_n}^{\alpha_n} \oplus c,$ the column $h=s_{\beta_1}^{\alpha_{1}}\oplus s_{\beta_2}^{\alpha_{2}}\oplus\cdots \oplus s_{\beta_n}^{\alpha_n} \oplus s_{j} \oplus c$ must burn as well for any $j \in \{1,2,\ldots,k\}$. This shows that any group element of the form of $g$ above will eventually be burned through repeated applications of this process, which implies that all of $\langle s_1,s_2,\ldots,s_k\rangle \oplus c$ will burn. 
    
    Since both row $r$ and column $h'$ are burned, the symbol $r\oplus h'$ also burns, and combining this with the burned column $c$ shows that the row $r\oplus h' \oplus c^{-1}$ must burn as well. From arguments made in the previous case, the column $c_j^*$ must be burned, and so combining this with row $r\oplus h' \oplus c^{-1}$ we have that the symbol $r\oplus h' \oplus c^{-1} \oplus c_j^*$ must burn. Finally, row $r$ and symbol $r\oplus h' \oplus c^{-1} \oplus c_j^*$ together imply that column $h' \oplus c^{-1} \oplus c_j^*$ must burn. Further, we have that $$h' \oplus c^{-1} \oplus c_j^* = s_{\beta_1}^{\alpha_{1}}\oplus s_{\beta_2}^{\alpha_{2}}\oplus\cdots \oplus s_{\beta_n}^{\alpha_n} \oplus c \oplus c^{-1} \oplus c_j^* = s_{\beta_1}^{\alpha_{1}}\oplus s_{\beta_2}^{\alpha_{2}}\oplus\cdots \oplus s_{\beta_n}^{\alpha_n} \oplus s_j \oplus c= h,$$ where the last equality follows from the definition of $s_j.$ The claim follows, and therefore, $$\langle s_1,s_2,\ldots,s_k\rangle \oplus c \subseteq \mathrm{Col}(S),$$ completing the proof. 
\end{proof}

We are now ready to determine the length of a shortest connected chain in a Latin square $L$ corresponding to a finitely generated group. As an immediate corollary, we obtain the lazy burning number of $H^L$. 

\begin{theorem}\label{thm:scc_group}
    Let $(G,\oplus) = \langle g_1,g_2,\ldots,g_{\alpha}\rangle$ be a group such that $\{g_1, g_2,\ldots, g_\alpha\}$ is a set of generators of $G$ of the smallest possible cardinality, and let $L(G)$ be the Cayley table of $G$ with the borders removed. We then have that $\mathrm{scc}(L(G)) = \alpha+1$. 
\end{theorem}

\begin{proof}
    We show that $\alpha+1$ is both a lower and upper bound for $\mathrm{scc}(L(G)),$ beginning with the lower bound. Let $L_0 \subseteq L_1 \subseteq \cdots \subseteq L_k = L(G)$ be a connected chain of $L$. By definition, for each $1 \leq i \leq k-1$, there exists a line $\ell_i$ such that $L_{i+1}$ is the smallest subsquare containing $L_i$ that intersects $\ell_i$. If $r$ and $c$ are the unique row and column of $L_1$, this implies that $L_{i+1}$ is the smallest subsquare intersecting the lines $\{r,c,\ell_1,\ell_2,\ldots,\ell_i\}$. Therefore, using Lemma~\ref{lem:coset_square}, we know that $\mathrm{Col}(L_k) = \langle c_1^*\oplus c^{-1},c_2^*\oplus c^{-1}\ldots,c_k^*\oplus c^{-1}\rangle \oplus c = G$, implying that $k-1 \geq \alpha$ and thus, that $\mathrm{scc}(L(G)) \geq \alpha+1$. 

    For the upper bound, let $c_{g_i}$ be the column of $L(G)$ corresponding to the group element $g_i$ and define a new sequence of nested subsquares of $L(G)$ denoted $S_i$, such that $S_i$ is the smallest subsquare intersecting the lines $\{r_e,c_e,c_{g_1},c_{g_2},\ldots,c_{g_{i-1}}\}$. It is evident that $S_0 \subseteq S_1 \subseteq \cdots \subseteq S_{\alpha+1}$, and by Lemma~\ref{lem:coset_square} we have that $\mathrm{Col}(S_{\alpha+1}) = G,$ implying that this forms a chain. Additionally, again by Lemma~\ref{lem:coset_square}, the line $c_{g_i}$ does not intersect $S_i$. We thus have that for each $1 \leq i \leq k-1$ there exists a line $\ell_i$ that does not intersect $S_i$ and such that $S_{i+1}$ is the smallest subsquare containing $S_i$ that intersects $\ell_i$. By definition, this chain is connected, and as the length of this chain is $\alpha+1$, we have that $\mathrm{scc}(L(G)) \leq \alpha +1$, completing the proof.
\end{proof}

The result below follows directly from Theorem~\ref{thm:scc_group}.

\begin{corollary}
\label{dim_plus_two}
 Let $(G,\oplus) = \langle g_1,\ldots,g_{\alpha}\rangle$ be a group and let $L(G)$ be the Latin square generated by $G$. We then have that $b_L(H^L) = \alpha+2$. 
\end{corollary}

\section{Further Directions}

A Steiner triple system, or STS, is a $3$-uniform hypergraph in which every pair of vertices appears together in exactly one hyperedge. Analogously to Latin squares, Steiner triple systems may be nested one inside the other; if an STS $G$ is a subhypergraph of another STS $H$, then we say that $G$ is a \emph{subsystem} of $H$. Burning Steiner triple systems was studied recently in \cite{our_STS_paper}, and reminiscent of Theorem~\ref{thm:3unistcupper}, the length of a shortest tight chain of subsystems is an upper bound on the lazy burning number of an STS. For this reason, we suspect that techniques from \cite{our_STS_paper} may apply in the context of Latin squares. 

The \emph{dimension} of an STS $H$ is the maximum number $d$ such that any set of $d$ vertices in $H$ is contained within a proper subsystem of $H$; see \cite{Valle_Dukes, Hilt_Tier}. In \cite{our_STS_paper}, it was shown that the lazy burning number of an STS is equal to one more than its dimension; the proof relies heavily on nested chains of subsystems that are tight in some sense, which is a concept that bears a resemblance to connected chains in Latin squares. We therefore ask if there is a natural definition for the dimension of a Latin square, and if so, does it relate to the lazy burning number of $H^L$? 

As we mainly focused on the lazy variant of burning, it would be interesting to explore the round-based version of burning on hypergraphs arising from Latin squares. It would also be interesting to explore both the lazy and round-based burning on other families of hypergraphs, such as those arising from mutually orthogonal Latin squares or Latin rectangles.

We explored a number of new structures related to Latin squares, such as connected chains of subsquares and cover-sequences. While these concepts were mainly used as tools to study lazy hypergraph burning, we think they should be promising to study in their own right. 

\section{Acknowledgements}
A grant from NSERC supported the first author. The second author acknowledges support from an NSERC CGS D scholarship.

\end{document}